\documentclass[12pt,a4paper]{article}

\usepackage{pgf,pgfarrows,pgfnodes,pgfautomata,pgfheaps,pgfshade}
\usepackage{amsmath,amssymb,amsfonts,amsthm,comment}
\usepackage[latin1]{inputenc}
\usepackage{colortbl}
\usepackage[english]{babel}

\usepackage{lmodern}
\usepackage[T1]{fontenc} 

\usepackage{times}
\usepackage{tikz}
\usepackage{bm}
\usepackage{color}

\usepackage{graphicx}

\usepackage{hyperref} 

\def\Z{\ifmmode{\mathbb Z}\else{$\mathbb Z$}\fi}
\def\C{\ifmmode{\mathbb C}\else{$\mathbb C$}\fi}
\def\Q{\ifmmode{\mathbb Q}\else{$\mathbb Q$}\fi}
\def\K{\ifmmode{\mathbb K}\else{$\mathbb K$}\fi}
\def\P{\ifmmode{\mathbb P}\else{$\mathbb P$}\fi}
\def\R{\ifmmode{\mathbb R}\else{$\mathbb R$}\fi}
\def\H{\ifmmode{\mathbb H}\else{$\mathbb H$}\fi}
\def\g{\ifmmode{\mathfrak g}\else {$\mathfrak g$}\fi}
\def\h{\ifmmode{\mathfrak h}\else {$\mathfrak h$}\fi}
\def\a{\ifmmode{\mathfrak a}\else {$\mathfrak a$}\fi}
\def\k{\ifmmode{\mathfrak k}\else {$\mathfrak k$}\fi}
\def\p{\ifmmode{\mathfrak p}\else {$\mathfrak p$}\fi}
\def\b{\ifmmode{\mathfrak b}\else {$\mathfrak b$}\fi}
\def\n{\ifmmode{\mathfrak n}\else {$\mathfrak n$}\fi}
\def\m{\ifmmode{\mathfrak m}\else {$\mathfrak m$}\fi}
\def\t{\ifmmode{\mathfrak t}\else {$\mathfrak t$}\fi}
\def\O{\ifmmode{\mathcal{O}}\else {$\mathcal{O}$}\fi}
\def\W{\ifmmode{\mathcal{W}}\else {$\mathcal{W}$}\fi}
\def\so{\ifmmode{\mathfrak {so}(n)} \else {$\mathfrak {so} (n)$}\fi}
\def\soc{\ifmmode{\mathfrak {so}(n,\C)} \else {$\mathfrak {so} (n,\C)$}\fi}
\def\u {\ifmmode{\mathfrak {u}(n)} \else {$\mathfrak {u} (n)$}\fi}
\def\su {\ifmmode{\mathfrak {su}(n)} \else {$\mathfrak {su} (n)$}\fi}
\def\sp {\ifmmode{\mathfrak {sp}(n)} \else {$\mathfrak {sp} (n)$}\fi}
\def\spc {\ifmmode{\mathfrak {sp}(n,\C)} \else {$\mathfrak {sp} (n,\C)$}\fi}
\def\slr {\ifmmode{\mathfrak {sl}(n,\R)} \else {$\mathfrak {sl} (n,\R)$}\fi}
\def\sl {\ifmmode{\mathfrak {sl}(n,\C)} \else {$\mathfrak {sl} (n,\C)$}\fi}
\def\slh {\ifmmode{\mathfrak {sl}(m,\H)} \else {$\mathfrak {sl} (n,\H)$}\fi}
\def\sup {\ifmmode{\mathfrak {su}(p,q)} \else {$\mathfrak {su} (p,q)$}\fi}
\def\gl {\ifmmode{\mathfrak {gl}(n,\R)} \else {$\mathfrak {gl} (n,\R)$}\fi}
\def\glc{\ifmmode{\mathfrak {gl}(n,\C)} \else {$\mathfrak {gl} (n,\C)$}\fi}
\def\glh{\ifmmode{\mathfrak {gl}(n,\H)} \else {$\mathfrak {gl} (n,\H)$}\fi}

\def\SLR{{SL(n,\R)}}
\def\SLH{{SL(n,\H)}}
\def\SLC{{SL(n,\C)}}
\def\SUP{{SU(p,q)}}

\newtheorem{thm}{Theorem}
\newtheorem{prop}[thm]{Proposition}
\newtheorem{defi}[thm]{Definition}
\newtheorem{lemma}[thm]{Lemma}
\newtheorem{cor}[thm]{Corollary}

\title{\hspace{-8pt} Schubert duality for $\SLR$-flag domains}
\author{Ana-Maria Brecan}
\date{}
\begin{document}
\maketitle
\begin{abstract}
\noindent
This paper is concerned with the study of spaces of naturally defined cycles associated to $SL(n,\R)$-flag domains. These are compact complex submanifolds in open orbits of real semisimple Lie groups in flag domains of their complexification. It is known that there are optimal Schubert varieties which intersect the cycles transversally in finitely many points and in particular determine them in homology. Here we give a precise description of these Schubert varieties in terms of certain subsets of the Weyl group and compute their total number. Furthermore, we give an explicit description of the points of intersection in terms of flags and their number.
\end{abstract}

\footnote{
Supported by SPP 1388 of the DFG.
}
\section{Background}
Here we deal with complex flag manifolds $Z=G/P$, with $G$ a complex semisimple Lie group and $P$ a complex parabolic subgroup, and consider the action of a real form $G_0$ of $G$ on $Z$. The real form $G_0$ is the connected real Lie group associated to the fixed point Lie algebra $\mathfrak{g}_0$ of an antilinear involution $\tau:\mathfrak{g}\rightarrow \mathfrak{g} $. It is known that $G_0$ has only finitely many orbits in $Z$ and therefore it has at least one open orbit. These and many other results relevant at the foundational level are proved in \cite{Wolf1969}. They are also summarised in the research monograph \cite{Fels2006}.
 Our work here is motivated by recent developments in the theory of cycle spaces of such a flag domain D. These arise as follows. Consider a choice $K_0$ of a maximal compact subgroup of $G_0$, i.e. $K_0$ is given by the fixed point set of a Cartan involution $\theta:G_0\rightarrow G_0$. Then $K_0$ has a unique orbit in $D\subset Z$, denoted by $C_0$, which is a complex submanifold of $D$. Equivalently, if $K$ denotes the complexification of $K_0$, one could look at $C_0$ as the unique minimal dimensional closed $K$-orbit in $D$. If $\dim C_0=q$, then $C_0$ can be regarded as a point in the Barlet space associated to $D$, namely $C^q(D)$, \cite{Barlet1975}. By definition, the objects of this space are formal linear combinations  $C=n_1C_1+\dots +n_kC_k$, with positive integral coefficients, where each $C_j$ is an irreducible $q$-dimensional compact subvariety of $D$. In this context $C_0$ is called the \textit{base cycle} associated to $K_0$. It is known that $C^q(D)$ is smooth at $C_0$ and thus one can talk about the irreducible component of $C^q(D)$ that contains $C_0$.  
\noindent
\bigskip
\newline
It is a basic method of Barlet and Koziarz, \cite{Barlet2000}, to transform functions on transversal slices to the cycles to functions on the cycle space. In the case at hand these transversal slices can be given using a special type of Schubert varieties which are defined with the help of the Iwasawa decomposition of $G_0$ (see part II of \cite{Fels2006} and the references therein). Recall that this is a global decomposition that exhibits $G_0$ as a product $K_0A_0N_0$, where each of the members of the decomposition are closed subgroups of $G_0$, $K_0$ is a maximal compact subgroup and $A_0N_0$ is a solvable subgroup. The Iwasawa decomposition is used to describe a type of Borel subgroups of $G$ which in a sense are as close to being real as possible. 
We define an \textit{Iwasawa-Borel} subgroup $B_I$ of $G$ to be a Borel subgroup that contains an Iwasawa-component $A_0N_0$ and we call the closure of an orbit of such a $B_I$ in $Z$ an \textit{Iwasawa-Schubert variety}. The Iwasawa-Borel subgroup can be equivalently obtained at the level of complex groups as follows. If $(G,K)$ is a symmetric pair, i.e. $K$ is defined by a complex linear involution, and $P=MAN$ where as usual $M$ is the centraliser of $A$ in $K$, then any such $B_I$ is given by choosing a Borel subgroup in $M$ and adjoining it to $AN$.

\noindent
\bigskip
\newline
The following result, \cite[p.101-104]{Fels2006}, has provided the motivation for our work.
\begin{thm}
If $S$ is an Iwasawa-Schubert variety such that $\dim S+ \dim C_0=\dim D$ and $S\cap C_0\ne \emptyset$ then the following hold:
\begin{enumerate}
\item{S intersects $C_0$ in only finitely many points $z_1,\dots, z_{d_s}$,}
\item{For each point of intersection the orbits $A_0N_0.z_j$ are open in $S$ and closed in $D$,}
\item{The intersection $S\cap C_0$ is transversal at each intersection point $z_j$.}
\end{enumerate}
\end{thm}
\noindent
For the next steps in this general area, for example computing in concrete terms the trace transform indicated above, we feel that it is important to understand precisely the combinatorial geometry of this situation. This means in particular to describe precisely which Schubert varieties intersect the base cycle $C_0$, their points of intersection and the number of these points. In particular such results will describe the base cycle $C_0$ (or any cycle in the corresponding cycle space) in the homology ring of the flag manifold $Z$. 
We have done this for the classical semisimple Lie group $\SLC$ and its real forms $\SLR$, $\SUP$ and $\SLH$ using methods which would seem sufficiently general to handle all classical semisimple Lie groups. 
\noindent
\bigskip
\newline
The present paper deals with the real form $\SLR$, while the results for the other two real forms can be found in the author's thesis, \cite{Ana}.
The description of the Schubert varieties is formulated combinatorially in terms of elements of the Weyl group of $G$. Interesting combinatorial conditions arise and the tight correspondence in between combinatorics and geometry is made explicit. 
Here, up to orientation we have only one open orbit. In the case of the full flag manifold, the Weyl group elements that parametrize the Schubert varieties of interest can be obtained from a simple game that chooses pairs of consecutive numbers from the ordered set $\{1,\dots, n\}$. Moreover, their total number is also a well-known number, the double-factorial. Surprisingly, the number of intersection points with the base cycle does not depend on the Schubert variety but only on the dimension sequence and in each case it is $2^{n/2}$. The main results are presented in \textbf{Theorem 8} and \textbf{Theorem 10}. In the case of the partial flag manifold the results depend on whether the open orbit is measurable or not. In the measurable case the main results are found in \textbf{Theorem 14} and \textbf{Theorem 16}, and in the non-measurable case in \textbf{Proposition 19}.

\section{The case of the real form $\SLR$}
\subsection{Preliminaries.}
Let $G=\SLC$ and $P$ be a parabolic subgroup of $G$ corresponding to a dimension sequence $d=(d_1,\dots, d_s)$ with $d_1+\dots + d_s=n$, i.e. $P$ is given by block upper triangular matrices of sizes $d_1$ up to $d_s$, respectively. Recall that in this case the flag manifold $Z=G/P$ can be identified with the set of all \textit{partial flags of type} $d$, namely $\{V: 0\subset V_1\subset \dots \subset V_s=\C^n\}$, where $\dim(V_i/V_{i-1})=d_i,\, \forall 1\le i \le s,$ $\dim V_0=0$. Equivalently, $Z$ can be defined with the help of the sequence $\delta=(\delta_1,\dots, \delta_s)$, with $\delta_i:=\sum_{k=1}^id_k=\dim V_i$, for all $1\le i \le s$. If $(e_1,\dots, e_n)$ is the standard basis in $\C^n$, the flags consisting of subspaces spanned by elements of this basis are called \textit{coordinate flags}. In the particular case when each $d_i=1$ we have a complete flag and the corresponding full flag variety is identified with the homogeneous space $\hat{Z}=G/B$, with $B$ the Borel subgroup of upper triangular matrices in $G$. In terms of the dimension sequence $d$ we have that $\dim Z=\sum_{1\le i < j \le s}d_id_j$. For each $d$ a fibration $\pi:\hat{Z}\rightarrow Z$ is defined by sending a complete flag to its corresponding partial flag of type $d$.
\noindent
\bigskip
\newline
Let us look at $\C^n$ equipped with the standard real structure $\tau: \C^n\rightarrow \C^n$, $\tau(v)=\overline{v}$ and the standard non-degenerate complex bilinear form $b:\C^n\times\C^n\rightarrow \C$, $b(v,w)=v^t\cdot w$ and view $G$ as the group of complex linear transformations on $\C^n$ of determinant $1$. Moreover, let $h:\C^n\times \C^n\rightarrow \C$ be the standard Hermitian form defined by $h(v,w)=b(\tau(v),w)=\overline{v}^t\cdot w$.
It follows that $G_0:=\{A\in G:\, \tau\circ A=A\circ \tau\}=\SLR$. If $\theta$ denotes both the Cartan involution on $G_0$ and on $G$ defined by $\theta(A)=(A^{-1})^t$, then $K_0:=SO(n,\R)$ and its complexification $K:=SO(n,\C)$ are both obtained as fixed points of the respective $\theta$'s. Fix the Iwasawa decomposition $G_0=\SLR=K_0A_0N_0$, where $A_0N_0$ are the upper triangular matrices with positive diagonal entries in $\SLR$. Thus, in this special case, the Iwasawa Borel subgroup $B_I$ is just the standard Borel subgroup of upper triangular matrices in $\SLC$.
\noindent
\bigskip
\newline
The following definitions give a geometric description in terms of flags of the open $G_0$-orbits in $Z$ and the base cycles associated to this open orbits. These results can be found in \cite{Huckleberry2001} and \cite{Huckleberry2002}.
\begin{defi}
A flag $z=(0 \subset V_{1}\subset \dots \subset V_{s}\subset \mathbb{C}^n)$ in $Z=Z_d$ is said to be $\tau$-\textbf{generic} if $dim(V_i\cap \tau(V_j))=max \{0,\delta_i+\delta_j-n\},\ \forall 1\le i,j\le s$. In other words, these dimensions should be minimal.
\end{defi}
\noindent
Note that in the case of $Z=G/B$ the condition of  $\tau$-genericity is equivalent to $$\tau(V_{j})\oplus V_{n-{j}}=\mathbb{C}^n, \quad \forall 1\le j \le [n/2].$$
\begin{defi}
A flag $z=(0 \subset V_{1}\subset \dots \subset V_{s}\subset \mathbb{C}^n)$ in $Z=Z_d$ is said to be \textbf{isotropic} if either $V_i \subseteq V_j^{\perp}$ or $V_i^\perp \subseteq V_j,\, \forall 1\le i,j \le s$. In other words, $$dim(V_i\cap V_j^\perp)=min\{\delta_i,n-\delta_j\}.$$ 
\end{defi}
\noindent
Note that in the case of $Z=G/B$ and $m=[n/2]$, the isotropic condition on flags is equivalent to $V_i\subset V_i^{\perp}$ for all $1\le i\le m$, $V_m=V_m^\perp$, if $n$ is even and the flags $V_{n-i}$ are determined by $V_{n-i}=V_{i}^{\perp},$ for all $1\le i \le m$.
\noindent
\bigskip
\newline
If $n=2m+1$ the unique open $G_0$-orbit is described by the set of $\tau$-generic flags. If $n=2m$ a notion of orientation arises on $\C^n_{\R}$ that is independent on the choice of basis. Since $G_0$ preserves orientation in this case we have two open orbits defined by the set of positively oriented $\tau$-generic flags and by the set of negatively oriented $\tau-$generic flags. One can define a map that reverses orientation and interchanges the two open orbits. It is therefore enough to only consider the open orbit defined by the positively-oriented flags. In each of the open orbits the base cycle $C_0$ is characterised by the set of isotropic flags.   
\noindent
\bigskip
\newline
Finally, recall the definition of Schubert varieties in a general flag manifold $Z=G/P$ and a few interesting properties in order to establish notation. In general, for a fixed Borel subgroup $B$ of $G$, a $B$-orbit $\mathcal{O}$ in $Z$ is called a Schubert cell and the closure of such an orbit is called a Schubert variety. A Schubert cell $\mathcal{O}$ in $Z$ is parametrized by an element $w$ of the Weyl group of $G$ and $Z$ is the disjoint union of finitely many such Schubert cells. Furthermore, the integral homology ring of $Z$, $H_*(Z,\Z)$ is a free $\Z$-module generated by the set of Schubert varieties. 
\noindent
\bigskip
\newline
If $G=\SLC$ and $T$ is the maximal torus of diagonal matrices in $G$, then the Weyl group of $G$ with respect to $T$ can be identified with $\Sigma_n$, the permutation group on $n$ letters. Moreover, the complete coordinate flags in $G/B$ are in $1-1$ correspondence with elements of $\Sigma_n$. Given a complete coordinate flag $$<e_{i_1}>\subset \dots \subset <e_{i_1},e_{i_2},\dots, e_{i_k}>\subset \dots \subset \mathbb{C}^n,$$ one can define a permutation $w$ by $w(k)=i_k$ for all $k$ and viceversa. The complete coordinate flags are also in $1-1$ correspondence with permutation matrices in $GL(n,\C)$. Given an element $w\in \Sigma_n$ one obtains a permutation matrix with column $i^{th}$ equal to $e_{w(i)}$ for each $i$. 
For this reason we use the symbol $w$ for both an element of $\Sigma_n$ in one line notation, i.e. w(1)w(2)\dots w(n), or for the corresponding permutation matrix. It will be clear from the context to which kind of representation we are referring to. 
\noindent
\bigskip
\newline
The fixed points of the maximal torus $T$ in $G/B$ are the coordinate flags $V_w$ for $w\in W$ and $G/B$ is the disjoint union of the Schubert cells $\mathcal{O}_w:=BV_w$, where $w\in W$. Since $\SLC/B$ is isomorphic to $GL(n,C)/B'$, where $B'$ are the upper triangular matrices in $GL(n,\C)$ we have another useful way of visualising Schubert cells via their matrix canonical form. If $\mathcal{O}_w$ is a given Schubert cell parametrized by $w\in W$,  then it can be represented as the matrix $Bw$ in which the lowest nonzero entry in each column is a 1 (on the $i^{th}$ column the $1$ lies on row $w(i)$) and the entries to the right of each leading $1$ are all zero. What is basically done is filling the permutation matrix with $*$'s above each $1$, having in mind the rule that to the right of each $1$ the elements must be zero. This leads us to the observation that the dimension of the Schubert cell $\mathcal{O}_w$ is given by the number of inversions in the permutation $w$, that is the length of $w$, or the number of $*$'s in the matrix representation.  
\noindent
\bigskip
\newline
Schubert cells and varieties can also be defined in $G/P$. Those will be indexed by elements of the coset $\Sigma_n/{\Sigma_{d_1}\times \Sigma_{d_2}\times \dots \times \Sigma_{d_s}}$ and each right coset contains a minimal representative, i.e. a unique permutation $w$ such that $w(1)<\dots < w(d_1), w(d_1+1)<\dots <w(d_1+d_2), \dots , w(d_1+\dots d_{s-1}+1)<\dots<w(d_1+\dots+d_s)=w(n)$. The dimension of the Schubert cell $C_{wP}:=BwP/P$ is the length of the minimal representative $w$ and there is a unique lift to a Schubert cell in $G/B$ of the same dimension, namely $\mathcal{O}_w:=BwB/B$. When working with $\mathcal{O}_{wP}$ we can thus use the same matrix representation as for $\mathcal{O}_w$ and many times we will refer to $\mathcal{O}_{wP}$ just by $\mathcal{O}_w$.

\subsection {Dimension-related computations}
\noindent
It is important for our discussion to compute the dimension of the base cycle and of the respective dual Schubert varieties in both the case of $G/B$ and of $G/P$. In the case when $B$ is the standard Borel subgroup in $\SLC$ and $K=SO(n,\C)$, the cycle $C_0$ is a compact complex submanifold of $D$ represented in the form $C_0=K.z_0\cong K/(K\cap B_{z_0})$ for a base point $z_0\in D$, where $K\cap B_{z_0}$ is a Borel subgroup of $K$. Since $C_0$ is a complex manifold $$dim\,C_0=dim\,T_{z_0}C_0=dim\,\mathfrak{k}/\mathfrak{k}\cap\mathfrak{b}_{z_0},$$ where $\mathfrak{k}$ is the Lie algebra associated to $K$ and $\mathfrak{b}$ is the Borel subalgebra associated to $B$.  Thus in the case when $n=2m$, $\dim C_0=m^2-m$ and the Schubert varieties of interest must be of dimension $m^2$. If $n=2m+1$, then $\dim C_0=m^2$ and the Schubert varieties of interest are among those of dimension $m^2+m$.
\noindent
\bigskip
\newline
\subsection{Introduction to the combinatorics}
The next two sections give a full description of the Schubert varieties of interest that intersect the base cycle $C_0$, the points of intersection and their number, in the case of an open $\SLR$-orbit $D$ in $Z$.
The first case to be considered is the case of $Z=G/B$, where $$\mathcal{S}_{C_0}:=\{S_w \text{ Schubert variety }: dim S_w+dim C_0=dimZ \text{ and } S_w\cap C_0 \ne \emptyset \}.$$ In what follows we describe the conditions that the element $w$ of the Weyl group that parametrizes the Schubert variety $S_w$ must satisfy in order for $S_w$ to be in $S_{C_0}$. One of the main ingredients for this is the fact that $S_w\cap D \subset \mathcal{O}_w $ and the fact that if $S_w\cap D\ne \emptyset$, then $S_w\cap C_0\ne \emptyset$. Moreover, no Schubert variety of dimension less than the codimension of the base cycle intersects the base cycle. These are general results that can be found in \cite[p.101-104]{Fels2006}
\begin{defi}
A permutation $w=k_1\dots k_ml_*l_m\dots l_1$ is said to satisfy the \textbf{spacing condition} if $l_i<k_i,\,\forall 1\le i\le m,$ where $l_*$ is removed
from the representation in the case $n=2m$.
\end{defi}
\noindent
For example, $265431$ satisfies the spacing condition, while $261534$ does not satisfy the spacing condition.
\begin{defi}
A permutation $w=k_1\dots k_ml_*l_m\dots l_1$ is said to satisfy the \textbf{double box contraction} condition if $w$ is constructed by the \textbf{immediate predecessor algorithm}: 
\newline
 $\bullet$ Start by choosing $k_1$ and $l_1:=k_1-1$ from the ordered set $\{1,\dots, n\}$. If we have chosen all the numbers up to $k_i$ and $l_i$ then to go to the step $i+1$ we make a choice of $k_{i+1}$ and $l_{i+1}$ from the ordered set $\{1,\dots, n\}-\{k_1,l_1,\dots,k_i,l_i\}$ in such a way that $l_{i+1}$ sits inside the ordered set at the left of $k_{i+1}$. 
\end{defi}
\noindent
Remark that a permutation that satisfies the double box contraction automatically satisfies the spacing condition as well, but not conversely. For example, $256341$ satisfies both the double box contraction and consequently the spacing condition while $265431$ does not satisfy the double box contraction even though it satisfies the spacing condition.
\noindent
\bigskip
\newline
The next results are meant to establish a tight correspondence between the combinatorics of the Weyl group elements that parametrize the Schubert varieties in $\mathcal{S}_{C_0}$ and the geometry of flags that describe the intersection points. Namely, we prove that the spacing condition on Weyl group elements corresponds to the $\sigma-$generic condition on flags. Similarly, the double box contraction condition on Weyl group elements corresponds to the isotropic condition on flags.  
\subsection{Main results}
The first result of this section describes the Schubert varieties that intersect the base cycle independent of their dimension.
\begin{prop}
A Schubert Variety $S_w$ corresponding to a permutation $$w=k_1\dots k_ml_*l_m\dots l_1,$$ where $l_*$ is removed from the representation in the case $n=2m$, has non-empty intersection with $C_0$ if and only if $w$ satisfies the spacing condition. 
\end{prop}

\begin{proof}
We use the fact that if a Schubert variety intersects the open orbit $D$, then it also intersects the cycle $C_0$ and prove that $S_w$ contains a $\tau$-generic point if and only if $w$ satisfies the spacing condition.
\noindent
\bigskip
\newline
First assume that $l_i<k_i$ for $i\le m$, $n=2m$. Under this assumption we need to prove that $$\tau(V_i)\oplus V_{2m-i}=\mathbb{C}^n\quad \forall i\le m,$$ where $V$ is an arbitrary flag in $S_w$. This is equivalent to showing that the matrix formed from the vectors generating $\tau(V_i)$ and $V_{2m-i}$ has maximal rank . Form the following pairs of vectors $(v_i,\bar{v}_i)$ and the matrices $$[v_1\bar{v}_1\dots v_i\bar{v}_iv_{i+1}\dots v_{2m-i}],\quad \forall i\le m.$$ These are the matrices corresponding to $$\tau(V_i)=<\bar{v}_1,\dots , \bar{v}_i>$$ and $$V_{2m-i}=<v_1,\dots, v_i, v_{i+1},\dots v_{2m-i}>.$$ We carry out the following set of operations keeping in mind that the rank of a matrix is not changed by row or column operations. The initial matrix is the canonical matrix representation of the Schubert cell $\mathcal{O}_w$. For the step $j=1$, $l_1<k_1$ and the last column corresponding to $l_1$ is eliminated from the initial matrix. 
\noindent
\bigskip
\newline
Next denote by $c_h$ the $h^{th}$ column of this matrices and obtain the following:
\newline
On the second column $c_2$ zeros are created on all rows starting with $k_1$ and going down to row $l_1+1$. This is done by subtracting suitable multiplies of $c_2$ from the columns in the matrix having a $1$ on these rows and putting the result on $c_2$. For example to create a zero on the spot corresponding to $k_1$ the second column is subtracted from the first column and the result is left on the second column.   On row $l_1$ a $1$ is created by normalising. This $1$ is the only $1$ on row $l_1$, because the last column that contained a $1$ on row $l_1$ was removed from the matrix. We now want to create zeros on $c_j$ for $j>2$ on row $l_1$. It is enough to consider those columns which have $1$'s on rows greater than $l_1$, because the columns with $1$'s on rows smaller than $l_1$ already have zeros below them. Finally, subtract from these columns suitable multiplies of $ c_2$.  We thus create a matrix that represents points in the Schubert variety $k_1l_1k_2\dots k_ml_m\dots l_2$, which obviously has maximal rank.
\noindent
\bigskip
\newline
Assume by induction that we have created the maximal rank matrix corresponding to points in the Schubert variety $k_1l_1k_2l_2\dots k_jl_jk_{j+1}\dots k_ml_m\dots l_{j+1}$. To go to the step $j+1$ remove from the matrix the last column corresponding to $l_{j+1}$ and add in between $c_{2j+1}$ and $c_{2j+2}$ the conjugate of $c_{2j+1}$ and reindex the columns. 
\noindent
\bigskip
\newline
On column $c_{2(j+1)}$ zeros are created on all rows starting with $k_{j+1}$ and going down to $l_{j+1}+1$ by subtracting suitable multiplies of $c_{2(j+1)}$ from the columns in the matrix having a $1$ on this rows and putting the result on $c_{2(j+1)}$. Next a $1$ is created on row $l_{2(j+1)}$ by normalisation. Again this is the only spot on row $l_{2(j+1)}$ with value $1$, because the column which had a $1$ on this spot was removed from the matrix. By subtracting suitable multiplies of $ c_{2(j+1)}$ from columns having a $1$ on spots greater than $l_{(j+1)}$, we create zeros on row $l_{j+1}$ at the right of the $1$ on column $c_{2(i+1)}$. We thus obtain points in the maximal rank matrix of the Schubert variety indexed by $k_1l_1k_2l_2\dots k_jl_jk_{j+1}l_{j+1}\dots k_ml_m\dots l_{j+2}$.   
\noindent
\bigskip
\newline
At step $j=m$ we obtain points in the maximal rank matrix of the Schubert variety indexed by $k_1l_1k_2l_2\dots k_ml_m$.
 For the odd dimensional case we insert in the middle a column corresponding to $l_*$ and observe that this of course does not change the rank of the matrix.
\noindent
\bigskip 
\newline
Conversely, if the spacing condition is not satisfied let $i$ be the smallest such that $k_i<l_i$ and look at the matrix $\tau(V_i)\oplus V_{2m-i}$. Then use the same reasoning as above to create a $1$ on the spot in the matrix corresponding to row $l_j$ and column $2j$ for all $j<i$ using our chose of $i$. Now there is a $1$ on each row in the matrix except on row $l_i$. Column $c_{2i}$ and $c_{2i-1}$ both have a $1$ on position $k_i$ and zeros bellow. Since $l_i>k_i$ this implies that for each $j<i$ there exist a column in the matrix that has $1$ on row $j$. Using these $1$'s we begin subtracting $c_{2i}$ from suitable multiplies of each such column, starting with $c_{2i-1}$ then going to the one that has 1 on the spot $k_i-1$, then to the one the has $1$ on the spot $k_i-2$ and so on and at each step the result is left on $c_{2i}$. This creates zeros on all the column $c_{2i}$ and proves that the matrix does not have maximal rank.
\end{proof}
\noindent
\textbf{Remark.} For dimension computations, recall how to compute the length $|w|$, of an element $w$ of $\Sigma_n$. Start with the number $1$ and move it from its position toward the left until it arrives at the beginning and associate to this its distance $p_1$ which is the number of other numbers it passes. Then move $2$ to the left until it is adjacent from the right to $1$ and compute $p_2$ in the analogous way. Continuing on compute $p_i$ for each $i$ and then the length of $w$ is just $\sum p_i$. 
\begin{lemma}
If $w=k_1\dots k_ml_*l_m\dots l_1$ satisfies the spacing condition and $|w|=m^2$ for the even dimensional case, or $|w|=m^2+m$ for the odd dimensional case, then $l_1=k_1-1$.
\end{lemma}
\begin{proof}
Suppose that $l_1<k_1-1$. Then there exist $j>1$ such that $p_j=k_1-1$ sits on position $j$ inside $w$. If $p_j$ sits among the $l$'s,  then construct $\tilde{w}$ by making the transposition $(j,n)$ that interchanges $p_j$ and $l_1$. If $p_j$ sits among the $k$'s  then construct $\tilde{w}$ by interchanging $k_1$ with $p_j$. Observe that $\tilde{w}$ still satisfies the spacing condition and $|\tilde{w}|\le m^2-1$ since in the first case all elements smaller then $k_1-1$ (at least one element, namely $l_1$) do not need to cross over $k_1-1$ anymore. In the second case $p_j$ does not need to cross over $k_1$ anymore and since $p_j=k_1-1$, the elements that need to cross $k_1$ on position $j$ remain the same as  the elements that needed to cross $p_j$ in the initial permutation. But this contradicts the fact that no Schubert variety of dimension less than $m^2$ intersects the cycle.
\noindent
\bigskip
\newline
For the odd dimensional case just add $l_*$ to the representation and consider $|w|=m^2+m$. The only case remaining to be considered is that where $l_*=k_1-1$. In this case we interchange $l_*$ with $l_1$ and observe that this still satisfies the spacing condition and it is of dimension strictly smaller then $m^2+m$. As above this implies a contradiction.
\end{proof}

\begin{thm}
A Schubert variety $S_w$ belongs to $\mathcal{S}_{C_0}$ if and only if $w$ satisfies the double box contraction condition. In particular, in this case $w$ satisfies the spacing condition and $|w|=m^2$, for $n=2m$, and $|w|=m^2+m$. 
\end{thm}

\begin{proof}
We prove the theorem using induction on dimension. The notation $w=(k,l_*,l)$ is used to represent the full sequence  $w=k_1\dots k_ml_*l_m\dots l_1$.
\newline
By the lemma above we see that $l_1$ must be defined by $l_1:=k_1-1$. Remove $l_1$ and $k_1$ from the set $\{1,\dots, n\}$ to obtain a set $\Sigma$ with $n-2$ elements. Define a bijective map  $\phi:\Sigma\rightarrow \{1,2,\dots,n-2\}$ by $\phi(x)=x$ for $x<l_1$ and $\phi(x)=x-2$ for $x>k_1$. By induction one constructs all possible $\tilde{w}=(\tilde{k},\tilde{l}_*,\tilde{l}) \in\Sigma_{n-2}$ using the immediate predecessor algorithm. Each such permutation parametrizes a Schubert variety $S_{\tilde{w}}$ in $S_{\tilde{C}_0}$.
\noindent
\bigskip
\newline
Now return to the original situation by defining for each $\tilde{w}$ a corresponding $w\in \Sigma_n$ with $w(1)=k_1$, $w(n)=l_1$ and $w(i+1)=\phi^{-1}(\tilde{w}(i))$ for $1\le i\le n-2$. It is immediate that $w$ satisfies the double box contraction condition. Hence it remains to compute $|(k,l_*,l)|$.   Let $\tilde{p}_j$ be the distances for the permutation $|(\tilde{k},\tilde{l}_*,\tilde{l})|$. First consider those elements $\varepsilon$ of the full sequence $(k,l_*,l)$ which are smaller than $l_1$ in particular which are smaller than $k_1$. In order to move them to their appropriate position one needs the number of steps $\tilde{p}_\varepsilon$ to do the same for their associated point in $(\tilde{k},\tilde{l}_*,\tilde{l})$ plus $1$ for having to pass $k_1$. Thus in order to compute $|(k,l_*,l)|$ from $|(\tilde{k},\tilde{l}_*,\tilde{l})|$ we must first add $k_1-2$ to the former. Having done the above, we now move $l_1$ to its place directly to the left of $k_1$. This requires crossing $2m+1-(k_1-1)$ larger numbers in the odd dimensional case and $2m-(k_1-1)$ numbers in the even dimensional case. So together we have now added $2m$ in the odd dimensional case and $2m-1$ in the even dimensional case to $|(\tilde{k},\tilde{l}_*,\tilde{l})|$ and $|(\tilde{k},\tilde{l})|$, respectively. All other necessary moves are not affected by the transfer to $|(\tilde{k},\tilde{l}_*,\tilde{l})|$ and back. So for those elements we have $\tilde{p}_\varepsilon=p_\varepsilon$ and it follows that $$|(k,l_*,l)|=|(\tilde{k},\tilde{l}_*,\tilde{l})|+2m=(m-1)^2+(m-1)+2m=m^2+m,$$
in the odd dimensional case and $$|(k,l)|=|(\tilde{k},\tilde{l})|+2m-1=(m-1)^2+2m-1=m^2,$$
in the even dimensional case.
\end{proof}

\begin{cor}
The number of Schubert varieties that intersect the cycle is the double factorial $n!!=(n-1)\cdot(n-3)\cdot\dots\cdot 1$.
\end{cor}

\begin{proof}
Observe that $k_1$ can be arbitrary chosen from $\{2,\dots,n\}$ and once it is chosen $l_1$ is fixed. This amounts to $(n-1)$ possibilities for the placement of $k_1$ and $l_1$. Now remove $l_1$ and $k_1$ from $\{1,\dots, n\}$ to obtain the set $\Sigma$ with $n-2$ elements. As before define the bijective map  $\phi:\Sigma\rightarrow \{1,2,\dots,n-2\}$ by $\phi(x)=x$ for $x<l_1$ and $\phi(x)=x-2$ for $x>k_1$. This gives us the sequence $(\tilde{k},\tilde{l}_*,\tilde{l})$ that satisfies our induction assumption and we thus obtain $(n-2-1)\cdot(n-2-3)\dots1$ Schubert varieties. Returning to our original situation one obtains the desired result. 
\end{proof}
\noindent
The next theorem gives a geometric description in terms of flags of the intersection points. The complete flags describing the intersection points are obtained in the following way from the Weyl group element that parametrizes $S_w$. In the case when $n=2m+1$ and $w=k_1\dots k_m l_*l_m\dots l_1$ the points of intersection are given by the following flags 
\begin{equation}
\label{points}
\begin{gathered}
<(\pm i) e_{l_1}+e_{k_1}>\subset\dots \subset <(\pm i) e_{l_1}+e_{k_1},\dots, (\pm i) e_{l_m}+e_{k_m},e_{l_*}>\subset \\
 <(\pm i) e_{l_1}+e_{k_1},\dots, (\pm i) e_{l_m}+e_{k_m},e_{l_*}, e_{l_m}>\subset \dots \subset \C^n. 
\end{gathered}
\end{equation}
In the case when $n=2m$, the complete flags are given by the same expression with the exception that the span of $e_{l_*}$ is removed from the representation. Of course, one must not forget that the positively oriented flags correspond to one open orbit and the negatively oriented flags to the other, but since this are symmetric with respect to the map that reverses orientation we have the same number of intersection points with the base cycle independently of the chosen open orbit.
\begin{thm}
A Schubert variety $S_w$ in $\mathcal{S}_{C_0}$ intersects the base cycle $C_0$ in $2^m$ points in the case when $n=2m+1$ and $2^{m-1}$ points in the case when $n=2m$. The points are given by \eqref{points}.
\end{thm}

\begin{proof}
Let $w=k_1\dots k_m l_*l_m\dots l_1$  and consider the canonical form of $\mathcal{O}_w$ given by a matrix $[v_1\dots v_n]$ with $$v_i=\alpha^i_{1}e_1+\dots +\alpha^i_{w(i)-1}e_{w(i)-1}+e_{w(i)}, \quad \forall 1\le i\le n.$$ The isotropic condition on flags translates to the fact that each matrix $[v_1\dots v_{n-i}]$ has the column vector $v_i$ perpendicular to itself and all the other vectors in the matrix. Such a matrix exists for each $i\le m$. Start with the initial matrix and for the first step $i=1$ disregard the last column of the initial matrix, for the second step $i=2$ disregard the last and the pre last column and go on until the step $i=m$ is reached. Looking at this process closely and imposing the isotropic conditions will give us an explicit description of all the intersection points.
\noindent
\bigskip
\newline
For the step $j=1$ the vector $v_n= \alpha^n_{1}e_1+\dots + e^n_{l_1}$, where $l_1=k_1-1$, is disregarded from the matrix. If $l_1=1$ then we are done, because $v_1=\alpha^1_1e_1+e_2$. From $v_1\cdot v_1=0$ it follows that $\alpha^1_1=\pm i$. If $l_1 > 1$, disregarding $v_n$  will create a matrix that has among its columns all vectors that contain a $1$ on entry $p$, where $1\le p<l_1$. Denote such column vectors with $f_p$. Using the relations $v_1\cdot f_1=0, \dots, v_1\cdot(f_{l_1-1})=0$, and computing step by step it follows that $\alpha^1_1=0,\dots, \alpha^1_{l_1-1}=0$. Now the only freedom left is on $\alpha_{l_1}$ and using $v_1\cdot v_1=0$ it follows that $(\alpha^1_{l_1})^2+1=0$. Therefore $$v_1=(\pm i)e_{l_1}+e_{k_1}.$$ The condition $v_1\cdot v_p=0$, for all $2\le p \le n-1$, is equivalent to $\alpha^p_{l_1}\cdot (\pm i)=0$, for all $2\le p \le n-1$, which is further equivalent to $\alpha^p_{l_1}=0$, for all $2\le p\le n-1$. The elements $\alpha^n_p$, for $1\le p \le l_1-1$, are all zero, because in the initial canonical representation of $\mathcal{O}_w$ all columns with a $1$ on the spot $p$, for $1\le p\le l_1-1$, sit before $v_n$ in the matrix and at the right of each such entry the row is completed with zeros.  
\noindent
\bigskip
\newline
For $j=2$, $v_n$ and $v_{n-1}$ are removed from the initial matrix. From the \textit{immediate predecessor} algorithm it follows that either $l_2=k_2-1$ or $l_2=l_1-1$ and $\alpha^{2}_{k_1}=0$. From the previous step $\alpha^{2}_{l_1}=0$. Therefore, even though in this step $v_n$ was removed from the matrix a zero was already created on row $l_1$ in $v_2$ in the previous step. Following the same algorithm as before we create zeros step by step starting with $v_2\cdot f_1=0$ and going further to  $v_2\cdot f_{p}=0$, where $f_p$ is the column where a $1$ sits on row $p$ for all $1\le p\le l_2-1$ exempt of $l_1$ and $k_1$ on which spots the values are already zero. The only freedom that remains is on the spot corresponding to $l_2$. Here, using $v_2\cdot v_2=0$, it follows that $$v_2=(\pm i)e_{l_2}+e_{k_2}.$$
\noindent
The condition $v_2\cdot v_p=0$, for all $3\le p \le n-2$, is equivalent with $\alpha^p_{l_2}\cdot (\pm i)=0$, for all $3\le p \le n-2$ which is further equivalent with $\alpha^p_{l_2}=0$, for all $3\le p\le n-2$. The elements $\alpha^{n-1}_p$, for $1\le p \le l_2-1$, are all zero. The reason for this is that in the initial canonical representation of $\mathcal{O}_w$ all columns with a $1$ on the spot $p$, for $1\le p\le l_2-1$, (with the exception of $p=l_1$ where we have created a zero already) sit before $v_{n-1}$ in the matrix. Moreover, at the right of each such entry the row is completed with zeros. 
\noindent
\bigskip
\newline
Assume that we have shown that $v_{s}=\pm i e_{l_s}+e_{k_s}$, for all $1\le s\le j-1$, with $j\le m$, that $\alpha_{l_s}^p=0$, for all $s+1\le p \le n-s$, and that $\alpha_p^{n-s+1}=0$, for all $1\le p \le l_{s}-1$.
To go to step $j$ one usees the fact that $l_j$ belongs to the set $$\{{k_j-1,l_1-1,\dots, l_{j-1}-1}\}$$ and repeats the procedure. In this case the columns $v_s$ for $n-j+1\le s \le n$ are removed from the matrix. As before zeros are created step by step starting with $v_j\cdot f_1$ and going up to $v_j\cdot f_p$, where $f_p$ is the column where a $1$ sits on row $p$ for all $1\le p \le l_j-1$, except of course when $p \in \{k_1,\dots, k_{j-1},l_{1},\dots {l_{j-1}}\}$ in which case even though the columns corresponding to this elements are not in the matrix their spots in column $j$ were already made zero in the previous step.
\noindent
\bigskip
\newline
It thus follows that the points of intersection $<v_1>\subset\dots \subset \C^n$ are given by the following flags $<(\pm i) e_{l_1}+e_{k_1}>\subset\dots \subset <(\pm i) e_{l_1}+e_{k_1},\dots, (\pm i) e_{l_m}+e_{k_m},e_{l_*}>\subset <(\pm i) e_{l_1}+e_{k_1},\dots, (\pm i) e_{l_m}+e_{k_m},e_{l_*}, e_{l_m}>\subset \dots \subset \C^n $
\end {proof}
\noindent
It thus follows that the homology class $[C_0]$ of the base cycle inside the homology ring of $Z$ is given in terms of the Schubert classes of elements in $\mathcal{S}_{C_0}$ by: $$[C_0]=2^m\sum_{S\in\mathcal{S}_{C_0}}\,[S], \text{ if } n=2m+1, \text{ and }[C_0]=2^{m-1}\sum_{S\in\mathcal{S}_{C_0}}\,[S], \text{ if } n=2m.$$
\noindent
\bigskip
\newline
\subsection{Main results for measurable open orbits in $Z=G/P$}
This section treats the case when $D$ is an open orbit in $Z=G/P$ and $C_0$ is the base cycle in $D$. Recall the notation
$$\mathcal{S}_{C_0}:=\{S_w \text{ Schubert variety }: dim S_w+dim C_0=dimZ \text{ and } S_w\cap C_0 \ne \emptyset \}.$$
The main idea is to lift Schubert varieties $S_w\in \mathcal{S}_{C_0}$ to minimal dimensional Schubert varieties $S_{\hat{w}}$ in $\hat{Z}:=G/B$ that intersect the open orbit $\hat{D}$ and consequently the base cycle $\hat{C}_0$. 
\noindent
\bigskip
\newline
The first step is to consider the situation when $D$ is a measurable open orbit. 
There are many equivalent ways of defining measurability in general. In our context however, this depends only on the dimension sequence $d$ that defines the parabolic subgroup $P$. Namely, an open $SL(n,\R)$ orbit $D$ in $Z=G/P$ is called \textit{measurable} if $P$ is defined by a symmetric dimension sequence as follows:
\begin{itemize}
\item{$d=(d_1,\dots, d_s,d_s,\dots, d_1)$ or $e=(d_1.\dots, d_s,e',d_s \dots d_1 )$, for $n=2m$, }
\item{$e=(d_1.\dots, d_s,e',d_s \dots d_1)$, for $n=2m+1$.}
\end{itemize}  
A general definition of measurability can be found in \cite{Fels2006} and the result used here can be found in \cite{Huckleberry2002}.
\noindent
\bigskip
\newline
As discussed in the preliminaries, a Schubert variety $S_w$ in $Z$ is indexed by the minimal representative $w$ of the parametrization coset associated to the dimension sequence defining $P$. Corresponding to the symmetric dimension sequence $d$, each such $w$ can be divided into blocks $B_j$ and $\tilde{B}_j$, both having the same number of elements $d_j$ for $1\le i \le s$. For example, $B_1=(w(1)<w(2)<\dots <w(d_1))$ and $\tilde{B}_1=(w(d_1+\dots +d_s+d_s+\dots+d_2+1)<\dots <w(n))$. The pairs $(B_j,\tilde{B}_j)$ are called \textbf{symmetric block pairs}. In the case of the symmetric dimension symbol $e$, one single block $B_{e'}$ of length $e'$ is introduced in the middle of $w$. In what follows $w$ will always correspond to a symmetric dimension sequence and it will satisfy the conditions of a minimal representative.

\begin{defi}
A permutation $w$ is said to satisfy the \textbf{generalized spacing condition}, if for each symmetric block pair $(B_j,\tilde{B}_j)$ the elements of $\tilde{B}_j$ can be arranged in such a way that if the elements of $B_j$ are denoted by $k_1^j\dots k_{d_j}^j$ and the rearranged elements of $\tilde{B}_j$ by $l_{1}^j\dots l_{d_j}^j$, then $l_i^j<k_i^j$ for all $1\le i \le d_j$. 
\end{defi}  
\noindent
Observe that in the case of the symmetric dimension sequence $e$ we can always rearrange the elements in the single block $B_{e'}$ so that they satisfy the spacing condition inside the block.

\begin{defi}
A permutation $B_1\dots B_sB_{e'} \tilde{B}_s\dots \tilde{B_1}$ is said to satisfy the \textbf{generalized double box contraction} condition if $w$ is constructed by the \textbf{generalized immediate predecessor algorithm}: 
\newline
 $\bullet$ The first symmetric block pair $(B_1=(k_i^1),\tilde{B}_1=(l_i^1))$ is constructed by choosing $d_1$ pairs $(l_i^1,k_i^1)$ of consecutive numbers from the ordered set $\{1,\dots, n\}$  
  \newline
 $\bullet$ If all symmetric block pairs up to $(B_j=(k_i^{j}),\tilde{B}_j=(l_i^{j}))$ have been chosen, then to go to the step $j+1$ choose ${d_{j+1}}$ pairs $(l_{i}^{j+1},k_i^{j+1})$ in such a way that $l_i^{j+1}$ sits at the immediate left of $k_i^{j+1}$ in the ordered set $\{1,\dots, n\}-\{\cup_{i=1}^jB_i\}-\{\cup_{i=1}^j\tilde{B}_i\}$, for all $1\le i \le d_{j+1}$.
\end{defi}
\noindent
Observe that in the case of the symmetric symbol $e$ the elements in the single block $B_{e'}$ can always be rearranged so that they satisfy the double box contraction condition inside the block. The symbol $B_{e'}$ will always be written in a representation of $w$ and disregarded in the case of the symbol $d$.
\noindent
\bigskip
\newline
If $w$ satisfies the generalized spacing condition, $\tilde{w}$ denotes the permutation obtain from $w$ by replacing each block $\tilde{B}_j=l_1^j\dots l_{d_j}^j$ with a choice of rearrangement of its elements $\tilde{l}_1^j\dots \tilde{l}_{d_j}^j$ required so that $w$ satisfies the generalized spacing condition. Further inside each rearranged block $\tilde{B}_j$ in $\tilde{w}$, $\tilde{l}_1^j\dots \tilde{l}_{d_j}^j$  is rewritten as $\tilde{l}_{d_j}^j \tilde{l}_{d_j-1}^j\dots \tilde{l}_1^j.$ In the case of $B_{e'}$ being part of the representation of $w$ the following rearrangement is chosen : if $e'$ is even then $B_{e'}$ is rearranged as $l'_{e'/2+1}$ $l'_{e'/2+2}\dots l'_{e'}l'_1\dots l'_{e'/2-1}l'_{e'/2}$ and if $e'$ is odd then $B_{e'}$ is rearranged as $l'_{(e'+1)/2+1}$ $l'_{(e'+1)/2+2}\dots l'_{e'}l'_1\dots l'_{(e'+1)/2-1}l'_{(e'+1)/2}$. Observe that now $\tilde{w}$ is a permutation that satisfies the spacing condition for the $G/B$ case and it is of course also just another representative of the coset that parametrizes $S_w$.

\begin{prop}
A Schubert variety $S_w$ parametrized by the permutation $$w=B_1B_2\dots B_sB_{e'}\tilde{B}_s\dots \tilde{B}_2\tilde{B}_1$$ has non-empty intersection with $C_0$ if and only if $w$ satisfies the generalized spacing condition, i.e. if and only if there exists a lift of $S_w$ to a Schubert variety $\tilde{S}_{\tilde{w}}$ that intersects the base cycle in $Z=G/B$.
\end{prop}
\begin{proof}
If $w$ satisfies the generalized spacing condition, then by the above observation one can find another representative $\tilde{w}$ of the parametrization coset of $S_w$, that satisfies the spacing condition and thus a Schubert variety $S_{\tilde{w}}$ that intersects $\hat{C}_0$. Because the projection map $\pi$ is equivariant it follows that $\pi(S_{\tilde{w}})=S_w$ intersects $C_0$.
\newline
Conversely, suppose that $S_w\cap C_0 \ne \emptyset$. Then for every point $p \in S_w\cap C_0 $ there exist $\hat{p}\in S_{\tilde{w}} \cap \hat{C}_0$ with $\pi (\hat{p})=p$ and $\pi(S_{\tilde{w}})=S_w$ for some Schubert variety in $G/B$ indexed by $\tilde{w}$. It follows that $\tilde{w}$ satisfies the spacing condition and $w$ is obtained from $\tilde{w}$ by dividing $\tilde{w}$ into blocks $B_1\dots B_sB_{e'}\tilde{B}_s\dots \tilde{B}_1$ and arranging the elements in each such block in increasing order. This shows that $w$ satisfies the generalized spacing condition.  
\end{proof}
\noindent
If $w=B_{d_1}\dots B_{d_s}\tilde{B}_{d_s}\dots \tilde{B_{1}}$, with $\tilde{B}_{d_j}=l^j_1\dots l^j_{d_j}$ for each $1\le j\le s$, then let $\hat{w}:=B_{d_1}\dots B_{d_s}\tilde{C}_{d_s}\dots \tilde{C}_{d_1}$, where $\tilde{C}_{d_j}=l^j_{d_j}l^j_{d_j-1}\dots l^j_2l^j_1$ for each $1\le j\le s$. If $B_{e'}=l'_1\dots l'_{e'}$ is part of the representation of $w$, then let $\tilde{B}_{e'}$ be the single middle block in the representation of $\hat{w}$ defined by: $l'_{e'/2+1}l'_{e'/2+2}\dots l'_{e'}l'_1\dots l'_{e'/2-1}l'_{e'/2}$ if $e'$ is even and $l'_{(e'+1)/2+1}l'_{(e'+1)/2+2}\dots l'_{e'}l'_1\dots l'_{(e'+1)/2-1}l'_{(e'+1)/2}$ if $e'$ is odd. Call such a choice of rearrangement for $w$ a \textbf{canonical rearrangement}. 
\noindent
\bigskip
\newline
Note that if $S_w\in \mathcal{S}_{C_0}$ lifts to $S_{\hat{w}}$, such that $\hat{w}$ satisfies the double box contraction condition, then $\dim S_{\hat{w}}-\dim S_w=(\dim \hat{Z}-\dim \hat{C}_0)-(\dim Z-\dim C_0)=(\dim \hat{Z}-\dim Z)-(\dim \hat{C_0}-\dim C_0).$ Since $\pi$ is a $G_0$ and $K_0$ equivariant map, if $F$ denotes the fiber of $\pi$ over a base point $z_0\in C_0$, then the fiber of $\pi|_{\hat{C}_0}:\hat{C}_0\rightarrow C_0$ over $z_0$ is just $F\cap \hat{C}_0$ and $\dim S_{\hat{w}}-\dim S_w$ must equal $\dim F - \dim (F\cap \hat{C}_0)$. 
\noindent
\bigskip
\newline
As stated in the preliminaries in the case of $Z=G/B$ and $m=[n/2]$, the isotropic condition on flags is equivalent to $V_i\subset V_i^{\perp}$ for all $1\le i\le m$, $V_m=V_m^\perp$, if $n$ is even and the flags $V_{n-i}$ are determined by $V_{n-i}=V_{i}^{\perp},$ for all $1\le i \le m$. Thus in the case of the dimension sequence $d=(d_1,\dots, d_s,d_s,\dots,d_1)$, $\dim F- \dim (F\cap C_0)$ is equal to $2\sum_{i=1}^sd_i(d_i-1)/2-\sum_{i=1}^sd_i(d_i-1)/2$ which is equal to $\sum_{i=1}^sd_i(d_i-1)/2$. In the case when the dimension sequence is given by $e=(d_1,\dots, d_s, e', d_s,\dots, d_1 )$ and the base point $z_0$ contains a middle flag of length $e$, it remains to add to the above number the difference in between the dimension of the total fiber over this flag and the dimension of the isotropic flags in this fiber. This is just a special case of the $G/B$ case for a full flag of length $e'$ and the number is $e'(e'-1)/2-(e'/2)^2+(e'/2)$ which equals to $(e'/2)^2$ when $e'$ is even and $e'(e'-1)/2-[(e'-1)/2]^2$ which equals to $(e'-1)(e'+1)/2$, when $e'$ is odd. 
\noindent
\bigskip
\newline
Now if $w$ satisfies the generalized double box contraction condition, then $\hat{w}$ satisfies the double box contraction condition. 
Furthermore, by construction, it is immediate that if $w$ satisfies the generalized double box contraction condition, then $|w|=|\hat{w}|-\sum_{i=1}^sd_i(d_i-1)/2$. That is because the block $\tilde{B}_{d_j}$ is formed by arranging the block $\tilde{C}_{d_j}$ in increasing order and thus crossing $l_1^j$ over $d_{j}-1$ numbers, and more generally $l_i^j$ over $d_{j}-i$ numbers for all $1\le i \le d_j$. Similarly, if $e'$ is part of the representation of $w$ and $e'$ is even then $|w|=|\hat{w}|-\sum_{i=1}^sd_i(d_i-1)/2-(e'/2)^2$ and if $e'$ is odd then $|w|=|\hat{w}|-\sum_{i=1}^sd_i(d_i-1)/2-(e'-1)(e'+1)/4$. 
\begin{thm}
A permutation $w=B_1B_2\dots B_sB_{e'}\tilde{B}_s\dots \tilde{B}_2\tilde{B}_1$ satisfies the generalized double box contraction condition if and only if $w$ parametrizes an Iwasawa-Schubert variety $S_w\in \mathcal{S}_{C_0}$ and the lifting map $f:\mathcal{S}_{C_0}\rightarrow \mathcal{S}_{\hat{C}_0}$ defined by $S_w\mapsto S_{\hat{w}}$, with $\hat{w}$ the canonical rearrangement of $w$, is injective.
\end{thm}
\begin{proof}
If $w$ satisfies the generalized double box contraction condition, then by the above observation the canonical rearrangement $\hat{w}$ of $w$ is just another representative of the parametrization coset of $S_w$. Moreover, $\hat{w}$ satisfies the double box contraction condition and thus parametrizes a Schubert variety $S_{\hat{w}}$ that intersects $\hat{C}_0$. Because the projection map $\pi$ is equivariant it follows that $\pi(S_{\tilde{w}})=S_w$ intersects $C_0$. From the remarks before the statement of the theorem it follows that if $w$ satisfies the generalized double box contraction condition, then the difference in dimensions in between $S_{\hat{w}}$ and $S_w$ is achieved, and this happens when one needs to write the elements of each block $\tilde{B_i}$ in strictly decreasing order to form the block $\tilde{C}_j$.
\noindent
\bigskip
\newline
Conversely, suppose that $S_w\in \mathcal{S}_{C_0}$ but $w$ does not satisfy the generalize double box contraction condition. It then follows that there exists a first block pair $(B_j,\tilde{B}_j)$ and a first pair $(k_i^j,l_i^j)$, for some $i$ in between $1$ and $d_j$ such that $l_i^{j-1}$ sits at the immediate left of $l_{i}^j$ and $k_{i}^j$ sits at the immediate right of $k_{i-1}^j$ in the ordered set $\{1,\dots, n\}-\{\cup_{s=1}^{j-1}B_s\}-\{\cup_{s=1}^{j-1}\tilde{B}_s\}$. This means that when $w$ is lifted to $\hat{w}$, the place of $l_{i-1}^j$ and $l_i^j$ remain the same because otherwise $w$ will not satisfy the double box contraction condition. But this implies that the difference in between the length of $\hat{w}$ and the length of $w$ is strictly smaller than what the difference $\dim S_{\hat{w}}-\dim S_{w}$ should be.

\end{proof}
%
\begin{thm}
A Schubert variety $S_w$ in $S_{C_0}$ intersects the base cycle $C_0$ in $2^{d_1+\dots + d_s}$ points in the case where $w$ is given by a symmetric symbol $e=(d_1,\dots,d_s,e,d_s,$ $\dots,d_1)$ and $n=2m+1$, while in the even dimensional case we have $2^{m-1}$  points in the case $d=(d_1,\dots, d_s,d_s,\dots,d_1)$ and $2^{d_1+\dots d_s-1}$ in the case $e=(d_1,\dots,d_s,e,d_s,$ $\dots,d_1)$. 
\end{thm}

\begin{proof}
The result follows from the lifting principle, because the intersection points of $S_w$ can be identified in a one-to-one manner with a subset of the intersection points of $S_{\hat{w}}$. Thus the cardinality of this subset can be directly computed.
\end{proof}
\noindent
It thus follows that the homology class $[C_0]$ of the base cycle inside the homology ring of $Z$ is given in terms of the Schubert classes of elements in $\mathcal{S}_{C_0}$ by: $$[C_0]=2^{m-1}\sum_{S\in\mathcal{S}_{C_0}}\,[S], \text{ if } n=2m, \text{ and } Z=Z_d,$$ $$[C_0]=2^{d_1+\dots +d_s-1}\sum_{S\in\mathcal{S}_{C_0}}\,[S], \text{ if } n=2m, \text{ and } Z=Z_e,$$  $$[C_0]=2^{d_1+\dots +d_s}\sum_{S\in\mathcal{S}_{C_0}}\,[S], \text{ if } n=2m+1, \text{ and } Z=Z_e.$$

\subsection{Main results for non-measurable open orbits in $Z=G/P$}

The last case to be considered is the case of $Z=G/P$ and $D\subset Z$ a non-measurable open orbit, i.e. the dimension sequence $f=(f_1,\dots, f_u)$ that defines $P$ is not symmetric. Associated with the flag domain $D$ there exists its measurable model, a canonically defined measurable flag domain $\hat{D}$ in $\hat{Z}=G/\hat{P}$ together with the projection map $\pi: \hat{Z} \rightarrow {Z}$. If $\hat{z}_0$ is a base point in $\hat{C}_0$, $z_0=\pi(\hat{z}_0)$ and we denote by $\hat{F}$ the fiber over $z_0$, by $H_0$ and $\tilde{H}_0$ the isotropy of $G_0$ at $z_0$ and $\hat{z}_0$ respectively, then $\hat{F}\cap\hat{D}=H_0/\hat{H}_0$, is holomorphically isomorphic with $\mathbb{C}^k$, where $k$ is root theoretically computable. 
\noindent
\bigskip
\newline
Moreover, if one considers the extension of the complex conjugation $\tau$ from $\C^n$ to $\SLC$, defined by $\tau(s)(v)=\tau(s(\tau(v)))$, for all $s\in \SLC$, $v\in \C^n$, then one obtains an explicit construction of $\hat{P}$ as follows. Consider the Levi decomposition of $P$ as the semidirect product $P=L\rtimes U$, where $L$ denotes its Levi part and $U$ the unipotent radical. If $P^{-}=L\rtimes U^{-}$ denotes the opposite parabolic subgroup to P, namely the block lower triangular matrix with blocks of size $f_1,\dots f_u$, then $\hat{P}= P\cap \tau(P^{-})$. Furthermore, the parabolic subgroup $\tau(P^{-})$ has dimension sequence $(f_u,\dots, f_1)$. These results are proved in complete generality in \cite{Wolf1995}.
\noindent
\bigskip
\newline
Because $\pi$ is a $K_0$ equivariant map, the restriction of $\pi$ to $\hat{C}_0$ maps $\hat{C}_0$ onto $C_0$ with fiber $\hat{C}_0\cap(\hat{F}\cap\hat{D})$. In this case this fiber is a compact analytic subset of $\mathbb{C}^k$ and consequently it is finite, i.e. the projection $\pi|_{\hat{C}_0}:\hat{C}_0\rightarrow C_0$ is a finite covering map. Because $C_0$ is simply-connected it follows that:
\begin{prop}
The restriction map $$\pi|_{\hat{C}_0}:\hat{C}_0\rightarrow C_0$$
is biholomorphic. In particular, if $q$ and $\hat{q}$ denote the respective codimensions of the cycles, it follows that $\hat{q}=dim(\hat{F})+q$
\end{prop}   
\noindent
If we denote with $\mathcal{S}_{C_0}$ the set of Schubert varieties $S_w$ in $Z$ that intersect the base cycle $C_0$ and $\dim S_w+\dim C_0=\dim Z$, where $w$ is a minimal representative of the parametrization coset of $S_w$ and by $\mathcal{S}_{\hat{C}_0}$ the analogous set in $\hat{Z}$, then the above discussion implies the following:
\begin{prop}
The map $\Phi: S_{C_0}\rightarrow \pi^{-1}(S_{C_0})\subset S_{\hat{C}_0}$ is bijective.
\end{prop}
\noindent
If $d=(d_1,\dots,d_s,d_s,\dots,d_1)$ or $e=(d_1,\dots,d_s,e',d_s,\dots,d_1)$ is a symmetric dimension sequence, then one can construct another dimension sequence out of it, not necessarily symmetric, by the following method. Consider an arbitrary sequence $t=(t_1,\dots,t_p)$ such that each $t_i\ge 1$ for all $1\le i\le p$, at least one $t_i$ is strictly bigger than $1$ and $t_1+\dots +t_p=2s$ or $2s+1$ depending on wether one considers $d$ or $e$, respectively. Associated to $t$ the sequence $\delta=(\delta_1,\dots , \delta_p)$ is defined by $\delta_{j}:=\sum_{i=1}^{j}\,t_i$. With the use of $\delta$ the new dimension sequence $f_{\delta}=(f_{\delta_{1}},\dots,f_{\delta_{p}})$ is defined by $f_{\delta_{1}}:=\sum_{i=1}^{\delta_{1}}d_i$,   $$f_{\delta_{j}}:=\sum_{i=\delta_{j-1}+1}^{\delta_{j}}\,d_i, \text{ for all } 2\le j \le p.$$ 
\noindent
\bigskip
\newline
Because $\hat{P}$ is obtained as the intersection of two parabolic subgroups $P$ and $\tau(P^{-})$, it follows that the dimension sequence $f$ of $P$ is obtained as above, from the dimension sequence of $\hat{P}$, as $f_{\delta}$ for a unique choice of $t$. For ease of computation we do not break up anymore the dimension sequence of $\hat{P}$ into its symmetric parts and we simply write it as $d=(d_1,\dots,d_s)$, where $s$ can be both even or odd. 
Using the usual method of computing the dimension of $Z$ it then follows that $$\dim Z=\dim {\hat{Z}}-\sum_{t_j>1}\sum_{\delta_{j-1}+1\le h<g\le \delta_j}\,d_hd_g.$$ For example if $P$ corresponds to the dimension sequence $(2,4,3)$, then an easy computation with matrices shows that $\hat{P}$ corresponds to the dimension sequence $(2,1,3,1,2)$, $t=(1,2,2)$ and $\delta=(1,3,5)$. Moreover, $\dim Z=\dim{\hat{Z}}-1\cdot3-1\cdot2$.
\noindent
\bigskip
\newline
Given the sequence $f=f_{\delta}$, we are now interested in describing the set $\mathcal{S}_{C_0}$. Let $S_{\hat{w}}$ in  $\mathcal{S}_{\hat{C}_0}$ be the unique Schubert variety associated to a given $S_w \in \mathcal{S}_{C_0}$ such that  $\pi (S_{\hat{w}})=S_{w}$.  If $\hat{w}$ is given in block form by $B_1\dots B_s$, where here again the notation used does not take into consideration the symmetric structure of $\hat{w}$, then $w$ is given in block form by $C_1\dots C_p$ corresponding to the dimension sequence $f_\delta$. The blocks $C_j$ are given by $C_1=\bigcup_{i=1}^{\delta_{1}}\, B_{d_i}$ and $$C_j=\bigcup_{i=\delta_{j-1}+1}^{\delta_j}\, B_{d_i}, \text{ for all } 2\le j\le p,$$ arranged in increasing order. Moreover,
\begin{align*}
\dim S_w&=\dim Z-\dim C_0=\dim Z-\dim{\hat{C}_0} \\
&=\dim{\hat{Z}}-\sum_{t_j>1}\sum_{\delta_{j-1}+1\le h<g\le \delta_j}\,d_hd_g -\dim{\hat{C}_0} \\
&=\dim{S_{\hat{w}}}-\sum_{t_j>1}\sum_{\delta_{j-1}+1\le h<g\le \delta_j}\,d_hd_g.
\end{align*}
Finally, understanding what conditions $\hat{w}$ satisfies in order for the above equality to hold amounts to understanding the difference in length that the permutation $\hat{w}$ looses when it is transformed into $w$. If $C_{j}$ contains only one $B$-block from $\hat{w}$, i.e. $t_j=1$, then this is already ordered in increasing order and it does not contribute to the decrease in dimension. If $C_{j}$ contains more $B$-blocks, say $$C_j=\bigcup_{i=\delta_{j-1}+1}^{\delta_j}\, B_{d_i},$$ we start with the first block $B_{\delta_{j-1}+1}$ which is already in increasing order and bring the elements from $B_{\delta_{j-1}+2}$ to their right spots inside the first block. As usual to each number we can associate a distance, i.e. the number of elements it needs to cross  in order to be brought on the right spot, and we denote by $\alpha_{\delta_{j-1}+2}$ the sum of this distances. The maximum value that $\alpha_{\delta_{j-1}+2}$ can attain is when all the elements in the second block are smaller than each element in the first block, i.e. the last element in the second block is smaller than the first element in the first block. In this case $\alpha_{\delta_{j-1}+2}=d_{\delta_{j-1}+1}d_{\delta_{j-1}+2}$, the product of the number of elements in the first block with the number of elements in the second block. Next we bring the elements in the $3^{rd}$ block among the already ordered elements from the first and second block. Observe that the maximal value that $\alpha_{\delta_{j-1}+3}$ can attain is $d_{\delta_{j-1}+1}d_{\delta_{j-1}+3}+d_{\delta_{j-1}+2}d_{\delta_{j-1}+3}$ when the last element in the $3^{rd}$ block is smaller than all elements in the first two blocks. In general we say that the group of blocks used to form  $C_j$ is in \textbf{strictly decreasing order} if it satisfies the following: the last element of block $B_{i+1}$ is smaller than the first element of block $B_i$ for all $\delta_{j-1}+1\le i \le \delta_j-1$. Consequently, if one wants to order this sequence of blocks into increasing order one needs to cross over $$\sum_{\delta_{j-1}+1\le h<g\le \delta_j}\,d_hd_g.$$ Thus if all the blocks $C_j$ with $t_j>1$ among $w$ come from groups of blocks arranged in strictly decreasing order in $\hat{w}$, then $$|w|=|\hat{w}|-\sum_{t_j>1}\sum_{\delta_{j-1}+1\le h<g\le \delta_j}\,d_hd_g.$$
What was just proved is the following:
\begin{prop}
A Schubert variety $S_{\hat{w}}\in \mathcal{S}_{\hat{C}_0}$ is mapped under the projection map to a Schubert variety $S_{w}\in \mathcal{S}_{C_0}$ if and only if all the blocks $C_j$ with $t_j>1$ among $w$ come from groups of blocks arranged in strictly decreasing order in $\hat{w}$.
\end{prop}
\noindent
As an example consider the complex projective space $Z=\mathbb{P}_5$. The dimension sequence of the measurable model in this case is given by $d=(1,4,1)$ and the Schubert varieties in $\mathcal{S}_{\hat{C}_0}$ are parametrized by the following permutations: $(2)(3456)(1)$, $(3)(1456)(2)$, $(4)(1256)(3)$, $(5)(1236)(4)$, $(6)(1234)(5)$. The only permutation that satisfies the strictly decreasing order among the last two blocks is the permutation $(2)(3456)(1)$ and this gives the only Schubert variety in $\mathcal{S}_{C_0}$ parametrized by $(2)(13456)$. More generally, for $Z=\mathbb{P}_n$ we have only one element in $\mathcal{S}_{C_0}$ parametrized by $(2)(13\dots n+1)$. For a more complicated example and for the proof of the following remark, see \cite{Ana}.
\noindent
\bigskip
\newline
\textbf{Remark.} In the context of the previous proposition, if $\pi_S$ denotes the map that sends $S_{\hat{w}}$ to $S_{w}$ and $U$ denotes unipotent radical of $B_I$, then there exists a subgroup $N$ of $U$ which acts freely and transitively on the fibers of $\pi_S$ which are just the fibers of $\pi:\hat{D}\rightarrow D$. This holds for the measurable model of any flag domain of any real form $G_0$.
\subsection* {Concluding remark}
The cases of the other real forms, $SL(m,H)$ and $SU(p,q)$,
of $\SLC$ are handled in detail in the author's thesis \cite{Ana}.  Let us close
our discussion here by briefly commenting on these results.

\bigskip
\noindent
For $n=2m$, $SL(m,\H)$ is defined to be the group of operators
in $\SLC$ which commute with the antilinear map $j:\mathbb{C}^n\to \mathbb{C}^n$,
$v\mapsto J\bar{v}$, where $J$ is the usual symplectic matrix with $J^2=-\mathrm{Id}$.  For simplicity
we only comment here on the case where $Z=G/B$ is the full flag manifold (the other cases are handled
by the expected lifting procedures.).  In that case a certain parity condition on a permutation $\omega =(k,\ell)$
plays the role of the spacing condition used for $G_0=\SLR$.  This condition tightens
up to the strict parity condition where it is required that the $k_i$ are odd and $\ell_i=k_i+1$.  The main result
for $Z=G/B$ can be stated as follows: There is only one open orbit $D$, a Schubert variety (with $B$ an 
Iwasawa-Borel) has non-empty intersection with $C_0$ if and only if $\omega $ satisfies the parity condition
and intersects it in isolated points if and only if it satisfies the strict parity condition.  In the later case the 
intersection consists of exactly one point.

\bigskip
\noindent
The case of the Hermitian groups $\mathrm{SU}(p,q)$, $p+q=n$, seems to be vastly more complicated than the
others.   This is in particular due to the large number of open orbits.  For example, if $z_0$ is a full flag 
$V_1<V_2<\ldots <V_n$ in $Z=G/B$, then $D=G_0.z_0$ is open if and only if the restriction of the mixed signature
form $\langle \ ,\ \rangle_{p,q}$ to each $V_i$ is non-degenerate.  Thus the open orbits are parameterised by
pairs $(a,b)$ of integer vectors with $(a_i,b_i)$ being the signature of the restriction of the form to $V_i$.

\bigskip
\noindent
While our combinatorial conditions are in a certain sense analogous to those for the other real forms, here we only 
give algorithms for determining exactly which Iwasawa-Schubert varieties have non-empty intersection with
a given flag domain $D$ and which intersect $C_0$ in only isolated points. Analogous to the $SL(m,H)$-case,
in the later case there is exactly one point of intersection.  As we show in typical examples and by giving detailed
descriptions of $S_{C_0}$ for the Hermitian symmetric spaces, i.e., $Z=G/P$ with $P$ maximal, the algorithms are quite
efficient (see \cite{Ana}).

\newpage
\bibliographystyle{alphanum}
\bibliography{citations}
Ana-Maria Brecan\\
Jacobs University Bremen\\
Campus Ring 1\\
28759, Bremen, Germany\\
a.brecan@jacobs-university.de
\end{document}